\documentclass[11pt,a4paper]{article}

\usepackage{caption2}
\usepackage{CJK}
\usepackage{tabls}
\usepackage{bm}
\usepackage{multirow}
\usepackage{amssymb}
\usepackage{amsfonts}
\usepackage{mathrsfs}
\usepackage{empheq}
\usepackage{amsmath}
\usepackage{amsthm}
\usepackage{graphicx}
\usepackage{color}
\usepackage{indentfirst}
\usepackage{hyperref}
\usepackage{float}
\usepackage{cases}
\usepackage[titletoc]{appendix}
\usepackage{bm}

\usepackage[colorinlistoftodos,english]{todonotes}

\allowdisplaybreaks%

\def\i1n{i=1,\cdots,n}
\def\j1n{j=1,\cdots,n}
\def\ij1n{i,j=1,\cdots,n}

\def \i{\mathrm i}

 \numberwithin{equation}{section}
\theoremstyle{definition}

 \newtheorem{thm}{\indent Theorem}[section]
 
 \newtheorem{lem}{\indent Lemma}[section]

\theoremstyle{definition}
\theoremstyle{theorem}
\theoremstyle{lemma}

\newcommand{\md}{\mbox{d}}
\newcommand{\be}{\begin{equation}}
\newcommand{\ee}{\end{equation}}
\newcommand{\beq}{\begin{equation*}}
\newcommand{\eeq}{\end{equation*}}

\begin{document}

\title{Global well-posedness for 2D inviscid and resistive MHD system near an equilibrium}


\author{Yuanyuan Qiao\thanks{School of Mathematics Science, Fudan University, Shanghai, P. R. China (20110180038@fudan.edu.cn).}}

\date{}

\maketitle

\begin{abstract}
We study the global existence of classical solutions for two-dimensional incompressible MHD system with only magnetic diffusion. By using the time-weighted lower-order energy and uniformly bounded higher-order energy estimates, we prove the global existence result under the assumption that the initial magnetic field is close enough to a constant magnetic vorticity equilibrium.
\end{abstract}



\section{Introduction}\label{intro}

The 2D classical incompressible MHD system can be expressed as follows:
\be
\begin{cases}\label{eqs1}
{ \begin{array}{ll} u_t+u\cdot\nabla u-\mu\Delta u+\nabla (P+\frac{1}{2}|B|^2)=B\cdot\nabla B,
\quad (t,x)\in\mathbb{R}^+\times\mathbb{R}^2,\\
B_t+u\cdot\nabla B-\nu\Delta B=B\cdot\nabla u,\\
\nabla\cdot u=\nabla\cdot B=0,\\
t=0: u=u_0(x), B=B_0(x).
 \end{array} }
\end{cases}
\ee
Here $u=(u_1,u_2)^\top$, $B=(B_1,B_2)^\top$ and $P$ represent the velocity field,
magnetic field and scalar pressure of the fluid, respectively.
$\mu$ denotes the kinematic viscosity, $\frac{1}{\nu}$ the magnetic Reynolds number.

Magnetohydrodynamics (MHD) system describes the evolution of electrically conducting fluids, such as plasmas, liquid metals, and electrolytes in a magnetic field through the proper coupling of hydrodynamic equations and Maxwell's equations.
In addition to its widespread physical applications, the theoretical analysis of (\ref{eqs1}) holds significant importance in mathematics as well.
The main focus of this paper is on well-posed results in the two-dimensional case.

It is widely recognized that the viscous and resistive MHD system ($\mu>0$, $\nu>0$) is globally well-posed in two-dimensional space (see \cite{AP2008,DL1972, ST1983}).
For the inviscid and non-resistive case ($\mu=\nu=0$), Bardos, Sulem and Sulem \cite{BSS1988} proved the global well-posedness result for small initial $u\pm (B-e)$ in a certain weighted H\"{o}lder's space (see further results in \cite{CL2018}).

For the viscous and non-resistive MHD equations ($\mu>0$, $\nu=0$), the global existence of smooth solutions for all time remains a challenging problem even in the two-dimensional case for generic smooth initial data. A recent strategy has been to search for global solutions under a non-trivial background magnetic field.
Lin, Xu and Zhang \cite{LXZ2015} obtained the global well-posedness with small smooth initial data in the Lagrangian coordinates, under the assumption that the initial magnetic field is close enough to the constant equilibrium state $B^{(0)}=e_2$. Ren, Wu, Xiang and Zhang \cite{RWXZ2014}
examined the same issue in Eulerian coordinates and employed direct energy estimates (see \cite{HL2014,Z2014} for more results).
In addition, Zhang \cite{Z2016} considered another constant equilibrium state $B^{(0)}=(\frac{1}{\varepsilon}, 0)^\top$, commonly known as the large background magnetic field.
Very recently, Qiao and Zhou \cite{QZ2023} proved that the 2D incompressible viscous MHD system
without magnetic diffusion may still be stable near a non-constant equilibrium state: $B^{(0)}=(x_2, -x_1)^\top$.

For the inviscid and resistive MHD system ($\mu=0$, $\nu>0$, the case of our consideration),
Lei and Zhou \cite{LZ2009} proved the global existence of $H^1$ weak solution to the system $(\ref{eqs1})$
(also see \cite{CW2011}). However, the problem of uniqueness and regularity remains open.
Under certain symmetry assumptions, Zhou and Zhu \cite{ZZ2018} established the global existence of small smooth solutions to the 2D MHD system near a constant equilibrium state in the periodic domain.
They employed the time-weighted energy method, which can also be applied to other equations such as the Oldroyd-B model, Euler-Poisson system and compressible MHD system (see \cite{WZ2022,Z2018,ZZ2019} for more details).
Wei and Zhang \cite{WZ2020} recently proved the global existence of small smooth solutions in $H^4(\mathbb{T}^2)$ under periodic boundary condition $\int_{\mathbb{T}^2}B(x,t)\md x=0$, and there is no requirement for the initial magnetic field to be close to any equilibrium state.
Ye and Yin \cite{YY2022} conducted further research on the lower regularity of $s>2$ in $H^s(\mathbb{T}^2)$,
improving the result for $s=4$ in \cite{WZ2020}.

For more well-posed results of the three-dimensional MHD system, please refer to references
\cite{AZ2017,CZZ2022,HXY2018,L2015,LZ2014,PZZ2018,LZ2015,XZ2015}.

Our main result concerns the following incompressible two-dimensional inviscid and resistive MHD system:
\be
\begin{cases}\label{eqs2}
{ \begin{array}{ll} u_t+u\cdot\nabla u+\nabla (P+\frac{1}{2}|B|^2)=B\cdot\nabla B,
\quad (t,x)\in\mathbb{R}^+\times\mathbb{R}^2,\\
B_t+u\cdot\nabla B-\Delta B=B\cdot\nabla u, \\
\nabla\cdot u=\nabla\cdot B=0,\\
t=0: u=u_0(x), B=B_0(x).
 \end{array} }
\end{cases}
\ee
Clearly, $\big(u^{(0)},P^{(0)},B^{(0)}\big)=\big(0, -|x|^2, (x_2,-x_1)^\top\big)$ is a special solution to this system.
By setting $p=P-P^{(0)}$, $b=B-B^{(0)}$, we can obtain the corresponding perturbation system of $(u,p,b)$:
\be\label{eqs3}
\begin{cases}
{ \begin{array}{ll} u_t+u\cdot\nabla u+\nabla\pi=b\cdot\nabla b-\partial_\theta b,
\quad (t,x)\in\mathbb{R}^+\times\mathbb{R}^2,\\
b_t+u\cdot\nabla b-\Delta b+\nabla\psi=b\cdot\nabla u-\partial_\theta u, \\
\nabla\cdot u=\nabla\cdot b=0,\\
t=0: u=u_0(x), b=b_0(x),\\
 \end{array} }
\end{cases}
\ee
where $\pi=P+\frac{1}{2}|B|^2+\frac{1}{2}|x|^2-\phi$, $\phi$ and $\psi$ are scalar stream functions.
In fact, the two-dimensional Biot-Savart law, along with the notation $\nabla^\bot=(-\partial_2, \partial_1)$,
gives
\begin{eqnarray}\label{1.4}
u=\nabla^\bot\psi,\quad b=\nabla^\bot\phi.
\end{eqnarray}
Subsequently, it is straightforward to obtain $b\cdot\nabla B^{(0)}=\nabla\phi$ and $u\cdot\nabla B^{(0)}=\nabla\psi$. Lastly, we also verify
$B^{(0)}\cdot\nabla=x_2\partial_1-x_1\partial_2:=-\partial_\theta.$
Here and in the following, we use $\partial_i$ to denote $\partial_{x_i}$, $i\in{1,2}$.
The variables $\theta$ and $r$, which will be used later, represent spatial variables in polar coordinates.

Assuming the initial data has the following symmetries:
\begin{align}\label{sym}
\begin{aligned}
&u_{0,1}(x),\quad b_{0,2}(x)\mbox{ are odd functions of}\ x_1,\\
&u_{0,2}(x),\quad b_{0,1}(x)\mbox{ are even functions of}\ x_1,\\
&u_{0,2}(x),\quad b_{0,1}(x)\mbox{ are odd functions of}\ x_2,\\
&u_{0,1}(x),\quad b_{0,2}(x)\mbox{ are even functions of}\ x_2.
\end{aligned}
\end{align}
We then state our main result.
\begin{thm}\label{thm0.1}
Consider the system (\ref{eqs3}) with the initial data satisfying the condition (\ref{sym}).
If there exists a small enough constant $\varepsilon>0$ such that
\begin{eqnarray}\label{**}
\|u_0\|_{H^{2s+8}}+\|b_0\|_{H^{2s+8}}\leq\varepsilon,
\end{eqnarray}
then the system (\ref{eqs3}) has a global smooth solution, provided $s\geq 2$.
\end{thm}

Our theorem shows that the system (\ref{eqs2}) has a global smooth solution near the non-constant equilibrium state $B^{(0)}=(x_2,-x_1)^\top$. In other words, the non-constant background magnetic field  may enhance the smoothness and stability of system (\ref{eqs2}).

One of our primary challenges arises from the higher-order energy estimate.
It seems impossible to deal with the nonlinear term $u\cdot\nabla u$ due to the absence of viscosity in the velocity equation.
To overcome this difficulty, we shall transform the system (\ref{eqs3}) into a damped wave type system.
By taking the derivative with respect to $t$ and making several substitutions, we derive
\be\label{3.1}
\begin{cases}
{ \begin{array}{ll} u_{tt}-\Delta u_t-\partial_{\theta\theta}u+\nabla q_1=F_t-\Delta F-\partial_\theta G, \\
b_{tt}-\Delta b_t-\partial_{\theta\theta}b+\nabla q_2=G_t-\partial_\theta F,\\
\nabla\cdot u=\nabla\cdot b=0,
 \end{array} }
\end{cases}
\ee
where we denote $\partial_{\theta\theta}=\partial_\theta^2$ and
\begin{align*}
\begin{aligned}
q_1:=&\pi_t-\Delta\pi-\partial_\theta\psi,\\
q_2:=&\psi_t-\partial_\theta\pi,\\
F:=&b\cdot\nabla b-u\cdot\nabla u,\\
G:=&b\cdot\nabla u-u\cdot\nabla b.
\end{aligned}
\end{align*}
Compared to (\ref{eqs3}), the wave structure in (\ref{3.1}) exhibits more regularity properties of $u$. The linearized system:
\be\label{3.1*}
\begin{cases}
{ \begin{array}{ll} u_{tt}-\Delta u_t-\partial_{\theta\theta}u+\nabla q_1=0, \\
b_{tt}-\Delta b_t-\partial_{\theta\theta}b+\nabla q_2=0,\\
\nabla\cdot u=\nabla\cdot b=0.
 \end{array} }
\end{cases}
\ee
reveals that $u$ and $b$ share the same wave structure.
On one hand, this inspires us to perform energy estimation directly on the damped wave-type system (\ref{3.1}) to derive a decay estimate of $\partial_\theta u$: $\int_0^t(1+\tau)^2\|\partial_\theta u\|_{\dot{H}^3}^2\md\tau$.
On the other hand, the uniqueness of the solution implies that property (\ref{sym}) will persist in the time evolution (refer to \cite{PZZ2018, ZZ2018} for detailed information), which means the zeroth Fourier modes for both $u$ and $b$ are equal to zero:
\begin{align}\label{1.2}
\int_{0}^{2\pi}u(t,r,\theta)\md\theta=0, \quad\int_{0}^{2\pi}b(t,r,\theta)\md\theta=0.
\end{align}
Consequently, by utilizing the Poincar\'{e} inequality, we can acquire the crucial decay estimate: $$\int_0^t(1+\tau)^2\|u\|_{\dot{H}^3}^2\md\tau.$$
However, the nonlinear term $u\cdot\nabla u$ still remains uncontrollable when we seek solutions in the Sobolev space $H^{2s+8}$.
Instead, we aim to control the growth of such norms through the energy frame defined in the following section (for more detailed information, refer to section 3, particularly equation (\ref{d1})).

Another difficulty in our proof arises from the decay estimate process, which generates a new difficult term, for $1\leq m\leq s-1$:
\begin{align*}
\int_0^t(1+\tau)^2\langle\Delta^{m+1}u, \Delta^{m+1}(b\cdot\nabla b)\rangle\md\tau.
\end{align*}
When there is ($2m+3$)-order derivative on $b$, the lower-order energy including $e_0(t)$ and $e_1(t)$ is not enough to control this term.
By carefully balancing the time-weight $(1+t)^2$ and the index of the interpolation inequalities, we finally successfully close the energy frame. (for more details, please refer to equation (\ref{d2}) in Section 3).

The paper is organised as follows: Section 2 introduces two essential lemmas and outlines the framework for the energy estimates. Section 3 is devoted to estimating the energy functionals defined in section 2. Lastly, Section 4 presents the proof of Theorem \ref{thm0.1}.

Let's conclude this section by some notations.
Throughout this paper we will frequently use $\|f,g\|_{H^k}$ to represent $\|f\|_{k}+\|g\|_{H^k}$,
where $\|\cdot\|_{H^k}$ $(k\geq 0)$ is the standard Sobolev norm.
We use $C$ to denote a finite positive constant which may be vary from line to line,
but do not depend on particular solutions or functions.
Given non-negative quantities $a,b\geq 0$, we will use $a\lesssim b$ to denote $a\leq Cb$.
Finally, $\langle\cdot,\cdot\rangle$ denotes the The $L^2$ inner product.

\section{Preliminaries}

In this section, we introduce two lemmas that will be frequently used throughout the paper.
Let's first recall the classical product and commutator estimates.
\begin{lem}\label{lem2.1}\cite{KP1988, KPV1991}
For $k\geq 0$, the following inequalities hold:
\begin{align*}
\|\nabla^k(fg)\|_{L^2}\lesssim&(\|f\|_{L^\infty}\|g\|_{H^k}
+\|f\|_{H^k}\|g\|_{L^\infty}),\\
\|[\nabla^k,f]g\|_{L^2}\lesssim&(\|\nabla f\|_{L^\infty}\|g\|_{H^{k-1}}+\|f\|_{H^{k-1}}\|g\|_{L^\infty}),
\end{align*}
where $[\cdot,\cdot]$ stands for the Poisson bracket: $[a,b]=ab-ba$.
\end{lem}

Next, we introduce the following Poincar\'{e} type inequality, which plays a crucial role in the proof of our main theorem.
\begin{lem}\label{lem2.2}
For any function $f\in H^k(\mathbb{R}^2)$ $(k\in \mathbb{N})$ satisfying the condition
\begin{align}\label{001}
\frac{1}{2\pi}\int_0^{2\pi}f(r, \theta)\md\theta=0
\end{align}
it holds that
\begin{align}\label{002}
\|f\|_{H^k(\mathbb{R}^2)}\lesssim\|\partial_\theta f\|_{H^k(\mathbb{R}^2)}.
\end{align}
\end{lem}

\begin{proof}
On one hand, according to the condition (\ref{001}), the standard Poincar\'{e} inequality gives
\begin{align*}
\|f\|_{L^2(\mathbb{R}^2)}\lesssim\|\partial_\theta f\|_{L^2(\mathbb{R}^2)}.
\end{align*}
On the other hand, note that
$\int_0^{2\pi}\partial_r f(r, \theta)\md\theta=0$, one has
\begin{align*}
\|f\|_{\dot{H}^1(\mathbb{R}^2)}^2
=&\|\partial_r f\|_{L^2(\mathbb{R}^2)}^2
+\|\frac{1}{r}\partial_\theta f\|_{L^2(\mathbb{R}^2)}^2\\
\lesssim&\|\partial_r\partial_\theta f\|_{L^2(\mathbb{R}^2)}^2
+\|\frac{1}{r}\partial_{\theta\theta}f\|_{L^2(\mathbb{R}^2)}^2\\
=&\|\partial_\theta f\|_{\dot{H}^1(\mathbb{R}^2)}^2.
\end{align*}
Therefore, using the commutant relation $[\Delta, \partial_\theta]=0$, we can deduce
(\ref{002}).
\end{proof}

Finally, we will outline the framework for the energy estimates.
For $s\in\mathbb{N}$ and a sufficiently small constant $0<\sigma<1$, we take $\sigma=\frac{1}{20}$ and set
\begin{align}
\nonumber E_0(t)=&\sup_{0\leq\tau\leq t}(1+\tau)^{-\sigma}\big(\|u(\tau)\|_{H^{2s+8}}^2+\|b(\tau)\|_{H^{2s+8}}^2\big)\\
&+\int_0^t(1+\tau)^{-1-\sigma}\big(\|u(\tau)\|_{H^{2s+8}}^2
+\|b(\tau)\|_{H^{2s+8}}^2\big)\md\tau\label{2.1}\\
\nonumber&+\int_0^t(1+\tau)^{-\sigma}\|\nabla b(\tau)\|_{H^{2s+8}}^2\md\tau,\\
\nonumber e_0(t)=&\sum_{m=0}^{s}\large\{\sup_{0\leq\tau\leq t}\big(\|u_\tau(\tau),b_\tau(\tau)\|_{\dot{H}^{2m}}^2
+\|\partial_\theta u(\tau),\partial_\theta b(\tau)\|_{\dot{H}^{2m}}^2
+\|u(\tau),b(\tau)\|_{\dot{H}^{2m+2}}^2\big)\\
&+\int_0^t\big(\|u_\tau(\tau),b_\tau(\tau)\|_{\dot{H}^{2m+1}}^2
+\|\partial_\theta u(\tau),\partial_\theta b(\tau)\|_{\dot{H}^{2m+1}}^2\big)\md\tau\big\}\label{2.2},\\
\nonumber e_1(t)=&\sum_{m=1}^{s-1}\large\{\sup_{0\leq\tau\leq t}(1+\tau)^2\big(\|u_\tau(\tau),b_\tau(\tau)\|_{\dot{H}^{2m}}^2
+\|\partial_\theta u(\tau),\partial_\theta b(\tau)\|_{\dot{H}^{2m}}^2
+\|u(\tau),b(\tau)\|_{\dot{H}^{2m+2}}^2\big)\\
&+\int_0^t(1+\tau)^2\big(\|u_\tau(\tau),b_\tau(\tau)\|_{\dot{H}^{2m+1}}^2
+\|\partial_\theta u(\tau),\partial_\theta b(\tau)\|_{\dot{H}^{2m+1}}^2\big)\md\tau\big\}\label{2.3}.
\end{align}

\section{Energy estimates}

The purpose of this section is to establish the uniform boundness of (\ref{2.1})-(\ref{2.3}).
Let's begin by estimating the highest order energy $E_0(t)$. It shows that the highest order norm
$H^{2s+8}(\mathbb{R}^2)$ of $u(t,\cdot)$ and $b(t,\cdot)$ will grow in the time evolution.
\begin{lem}\label{lem3.1}
From the definition in (\ref{2.1}), for $s\geq 2$, we have
\begin{align}\label{1}
E_0(t)\lesssim\varepsilon^2+E_0(t)e_0(t)^{\frac{1}{4}}e_1(t)^{\frac{1}{4}}+E_0(t)e_0(t)^{\frac{1}{2}}.
\end{align}
\end{lem}
\begin{proof}
Applying $\Delta^{2s+8}$ to system (\ref{eqs3}) and taking inner product with $u$ and $b$, respectively.
Then, adding the results up and multiplying the time-weight $(1+t)^{-\sigma}$, we arrive at
\begin{align}\label{*}
\begin{aligned}
&\frac{1}{2}\frac{\md}{\md t}(1+t)^{-\sigma}(\|u\|_{\dot{H}^{2s+8}}^2+\|b\|_{\dot{H}^{2s+8}}^2)
+\frac{\sigma}{2}(1+t)^{-1-\sigma}(\|u\|_{\dot{H}^{2s+8}}^2+\|b\|_{\dot{H}^{2s+8}}^2)\\
&+(1+t)^{-\sigma}\|\nabla b\|_{\dot{H}^{2s+8}}^2\\
=&-(1+t)^{-\sigma}\langle\nabla^{2s+8}u,[\nabla^{2s+8}, u\cdot\nabla]u\rangle
+(1+t)^{-\sigma}\langle\nabla^{2s+8}u,[\nabla^{2s+8}, b\cdot\nabla]b\rangle\\
&-(1+t)^{-\sigma}\langle\nabla^{2s+8}b,[\nabla^{2s+8}, u\cdot\nabla]b\rangle
+(1+t)^{-\sigma}\langle\nabla^{2s+8}b,[\nabla^{2s+8}, b\cdot\nabla]u\rangle\\
:=&\sum_{i=1}^4I_i.
\end{aligned}
\end{align}
Here, we employed incompressibility conditions and the commutant relation $[\Delta, \partial_\theta]=0$.

For the term $I_1$, using H\"{o}lder's inequality, the commutator estimate, Sobolev imbedding theorem
and Lemma \ref{lem2.2}, we have
\begin{align*}
\begin{aligned}
|I_1|
\lesssim&(1+t)^{-\sigma}\|u\|_{\dot{H}^{2s+8}}^2\|\nabla u\|_{L^{\infty}}\\
\lesssim&(1+t)^{-\sigma}\|u\|_{\dot{H}^{2s+8}}^2\|\nabla u\|_{L^2}^{\frac{1}{2}}
\|\nabla^3 u\|_{L^2}^{\frac{1}{2}}\\
=&C(1+t)^{-\frac{1}{2}-\frac{\sigma}{2}}\|u\|_{\dot{H}^{2s+8}}
(1+t)^{-\frac{\sigma}{2}}\|u\|_{\dot{H}^{2s+8}}
\|\nabla u\|_{L^2}^{\frac{1}{2}}
(1+t)^{\frac{1}{2}}\|\nabla^3 u\|_{L^2}^{\frac{1}{2}}
\end{aligned}
\end{align*}
Hence,
\begin{align}\label{d1}
\begin{aligned}
\int_0^t|I_1|\md\tau
\lesssim& \sup_{0\leq\tau\leq t}(1+\tau)^{-\frac{\sigma}{2}}\|u\|_{H^{2s+8}}
\big(\int_0^t(1+\tau)^{-1-\sigma}\|u\|_{H^{2s+8}}^2\md\tau\big)^{\frac{1}{2}}\\
&\times\big(\int_0^t\|\partial_\theta u\|_{\dot{H}^1}^2\md\tau\big)^{\frac{1}{4}}
\big(\int_0^t(1+\tau)^2\|\partial_\theta u\|_{\dot{H}^3}^2\md\tau\big)^{\frac{1}{4}}\\
\lesssim&E_0(t)e_0(t)^{\frac{1}{4}}e_1(t)^{\frac{1}{4}}.
\end{aligned}
\end{align}

Similarly, for $I_2$, one has
\begin{align*}
\int_0^t|I_2|\md\tau
\lesssim&\int_0^t(1+\tau)^{-\sigma}\|u\|_{\dot{H}^{2s+8}}\|b\|_{\dot{H}^{2s+8}}\|\nabla \partial_\theta b\|_{H^2}\md\tau\\
\lesssim&\sup_{0\leq\tau\leq t}(1+\tau)^{-\frac{\sigma}{2}}\|u\|_{H^{2s+8}}
\big(\int_0^t(1+\tau)^{-\sigma}\|\nabla b\|_{H^{2s+7}}^2\md\tau\big)^{\frac{1}{2}}
\big(\int_0^t\|\nabla \partial_\theta b\|_{H^2}^2\md\tau\big)^{\frac{1}{2}}\\
\lesssim&E_0(t)e_0(t)^{\frac{1}{2}}.
\end{align*}

For $I_3$ and $I_4$, using H\"{o}lder's inequality and Lemma \ref{lem2.1}, we have
\begin{align*}
|I_3|+|I_4|
\lesssim&(1+t)^{-\sigma}(\|u\|_{\dot{H}^{2s+8}}\|b\|_{\dot{H}^{2s+8}}\|\nabla b\|_{L^{\infty}}
+\|b\|_{\dot{H}^{2s+8}}^2\|\nabla u\|_{L^{\infty}}).
\end{align*}
Thus, taking the $L^1$ norm over $[0,t)$ and employing similar techniques as to $I_1$ and $I_2$, we arrive at
\begin{align*}
\int_0^t|I_3|+|I_4|\md\tau
\lesssim E_0(t)e_0(t)^{\frac{1}{2}}.
\end{align*}

Integrating (\ref{*}) respect to time and inserting $I_1-I_4$, we complete the estimate of $E_0(t)$.
\end{proof}

The proof of the decay estimate $e_1(t)$ relies crucially on the following lemma.
Taking advantage of the damped wave type system (\ref{3.1}), we shall establish the lower-order energy estimate $e_0(t)$.
\begin{lem}\label{lem3.3}
From the definition in (\ref{2.2}), for $s\geq 2$ and $0<\sigma<1$, we have
\begin{align}\label{2}
\begin{aligned}
e_0(t)\lesssim&\varepsilon^2+e_0(t)^{\frac{3}{2}}
+E_0(t)^{\frac{1}{6}}e_0(t)^{\frac{7}{6}}e_1(t)^{\frac{1}{6}}
+E_0(t)^{\frac{1}{3}}e_0(t)e_1(t)^{\frac{1}{6}}
+E_0(t)^{\frac{2}{3}}e_0(t)^{\frac{2}{3}}e_1(t)^{\frac{1}{6}}\\
&+E_0(t)^{\frac{1}{2}}e_0(t)^{\frac{3}{4}}e_1(t)^{\frac{1}{4}}.
\end{aligned}
\end{align}
\end{lem}
\begin{proof}
Recall the damped wave type system (\ref{3.1}):
\begin{eqnarray*}
\begin{cases}
{ \begin{array}{ll} u_{tt}-\Delta u_t-\partial_{\theta\theta}u+\nabla q_1=F_t-\Delta F-\partial_\theta G, \\
b_{tt}-\Delta b_t-\partial_{\theta\theta}b+\nabla q_2=G_t-\partial_\theta F,
 \end{array} }
\end{cases}
\end{eqnarray*}
where
$F=b\cdot\nabla b-u\cdot\nabla u$,
$G=b\cdot\nabla u-u\cdot\nabla b$
are nonlinear terms.

Next, we will take five steps to this proof.\\
\textbf{Step 1}\quad

Taking the $L^2$ inner product of (\ref{3.1}) with $\Delta^{2m} u_t$, $\Delta^{2m} b_t$ respectively, and adding the results, one has
\begin{align}\label{3.2}
\begin{aligned}
&\frac{1}{2}\frac{\md}{\md t}\big(\|u_t,b_t\|_{\dot{H}^{2m}}^2+\|\partial_\theta u,\partial_\theta b\|_{\dot{H}^{2m}}^2\big)
+\|\nabla u_t, \nabla b_t\|_{\dot{H}^{2m}}^2\\
=&\langle\Delta^{2m}u_t, F_t-\Delta F-\partial_\theta G\rangle+\langle\Delta^{2m}b_t, G_t-\partial_\theta F\rangle
:=N_1.
\end{aligned}
\end{align}
\textbf{Step 2}\quad

To obtain the positive sign for
$\|\partial_\theta u, \partial_\theta b\|_{\dot{H}^{2m+1}}$, we repeat the $L^2$ inner product of (\ref{3.1}) with $-\Delta^{2m+1} u$ and $-\Delta^{2m+1} b$, respectively. Adding the results, we get
\begin{align}\label{3.3}
\begin{aligned}
&\frac{\md}{\md t}\big(-\langle\Delta^{2m+1}u, u_t\rangle
-\langle\Delta^{2m+1}b,b_t\rangle
+\frac{1}{2}\|u,b\|_{\dot{H}^{2m+2}}^2\big)
-\|u_t,b_t\|_{\dot{H}^{2m+1}}^2+\|\partial_\theta u,\partial_\theta b\|_{\dot{H}^{2m+1}}^2\\
=&-\langle\Delta^{2m+1}u, F_t-\Delta F-\partial_\theta G\rangle-\langle\Delta^{2m+1}b, G_t-\partial_\theta F\rangle
:=N_2.
\end{aligned}
\end{align}
\textbf{Step 3}\quad

Multiplying (\ref{3.2}) by a suitable large number and adding (\ref{3.3}),
thanks to Young's inequality, we deduce that
\begin{align}\label{3.5}
\begin{aligned}
\frac{\md}{\md t}
\big(\|u_t,b_t\|_{\dot{H}^{2m}}^2+\|\partial_\theta u,\partial_\theta b\|_{\dot{H}^{2m}}^2
+\|u_t, b_t\|_{\dot{H}^{2m+1}}^2\big)
+\|\partial_\theta u,\partial_\theta b\|_{\dot{H}^{2m+1}}^2
\lesssim N_1+N_2.
\end{aligned}
\end{align}
Here, we also used the following inequalities:
\begin{align*}
\langle\Delta^{2m+1}u, u_t\rangle
\leq& C\|u\|_{\dot{H}^{2m+2}}\|u_t\|_{\dot{H}^{2m}}
\leq \epsilon\|u\|_{\dot{H}^{2m+2}}^2+C(\epsilon)\|u_t\|_{\dot{H}^{2m}}^2,\\
\langle\Delta^{2m+1}b, b_t\rangle
\leq& C\|b\|_{\dot{H}^{2m+2}}\|b_t\|_{\dot{H}^{2m}}
\leq \epsilon\|b\|_{\dot{H}^{2m+2}}^2+C(\epsilon)\|b_t\|_{\dot{H}^{2m}}^2,
\end{align*}
for a certain small constant $\epsilon>0$. $C(\epsilon)$ is a positive constant depending only on $\epsilon$.\\
\textbf{Step 4}\quad

We shall estimate each term on the right-hand side of (\ref{3.5}).
First, we divide $N_1$ into two parts:
\begin{align*}
\begin{aligned}
N_{11}=&\langle\Delta^{2m}u_t, F_t-\partial_\theta G\rangle+\langle\Delta^{2m}b_t, G_t-\partial_\theta F\rangle,\\
N_{12}=&\langle\Delta^{2m}u_t, -\Delta F\rangle.
\end{aligned}
\end{align*}
For the first part of $N_1$, namely
\begin{align*}
\begin{aligned}
N_{11}
=&\langle\Delta^mu_t, \Delta^m(b\cdot\nabla b)_t\rangle
-\langle\Delta^mu_t,\Delta^m(u\cdot\nabla u)_t\rangle\\
&-\langle\Delta^mu_t, \Delta^m\partial_\theta(b\cdot\nabla u)\rangle
+\langle\Delta^mu_t,\Delta^m\partial_\theta(u\cdot\nabla b)\rangle\\
&+\langle\Delta^mb_t, \Delta^m(b\cdot\nabla u)_t\rangle
-\langle\Delta^mb_t,\Delta^m(u\cdot\nabla b)_t\rangle\\
&-\langle\Delta^mb_t, \Delta^m\partial_\theta(b\cdot\nabla b)\rangle
+\langle\Delta^mb_t,\Delta^m\partial_\theta(u\cdot\nabla u)\rangle,
\end{aligned}
\end{align*}
we only take care of the first term. In fact, since $u_t$, $b_t$, $\partial_\theta u$ and $\partial_\theta b$ have the same regularity on the left-hand side of (\ref{3.5}), the remaining terms of $N_{11}$ can be treated in the same way with some modifications.
To overcome the lack of the space-time $L^2$ norm for $u_t$, $b_t$, $\partial_\theta u$ and $\partial_\theta b$, we handle the first term in two cases: $m\geq 1$ and $m=0$.
For $m\geq 1$,
using the divergence-free condition $\nabla\cdot b=0$, the product estimate, Sobolev imbedding inequality and Lemma \ref{lem2.2}, we have
\begin{align*}
\langle\Delta^mu_t, \Delta^m(b\cdot\nabla b)_t\rangle
=&\langle\nabla^{2m}u_t, \nabla^{2m+1}(b\otimes b)_t\rangle\\
\lesssim&\|u_t\|_{\dot{H}^{2m}}\big(\|b_t\|_{L^\infty}\|b\|_{\dot{H}^{2m+1}}
+\|b\|_{L^\infty}\|b_t\|_{\dot{H}^{2m+1}}\big)\\
\lesssim&\|u_t\|_{\dot{H}^{2m}}\big(\|b_t\|_{H^2}\|\partial_\theta b\|_{\dot{H}^{2m+1}}
+\|\partial_\theta b\|_{H^2}\|b_t\|_{\dot{H}^{2m+1}}\big).
\end{align*}
Hence,
\begin{align*}
&\int_0^t\langle\Delta^mu_\tau, \Delta^m(b\cdot\nabla b)_\tau\rangle\md\tau\\
\lesssim&\sup_{0\leq\tau\leq t}\|b_\tau\|_{H^2}
(\int_0^t\|u_\tau\|_{\dot{H}^{2m}}^2\md\tau)^{\frac{1}{2}}
(\int_0^t\|\partial_\theta b\|_{\dot{H}^{2m+1}}\md\tau)^{\frac{1}{2}}\\
&+\sup_{0\leq\tau\leq t}\|\partial_\theta b\|_{H^2}
(\int_0^t\|u_\tau\|_{\dot{H}^{2m}}^2\md\tau)^{\frac{1}{2}}
(\int_0^t\|b_\tau\|_{\dot{H}^{2m+1}}\md\tau)^{\frac{1}{2}}\\
\lesssim& e_0(t)^{\frac{3}{2}},
\end{align*}
where we have used the fact that
$$\|u_t\|_{\dot{H}^{2m}}^2\lesssim\|u_t\|_{\dot{H}^{2m-1}}\|u_t\|_{\dot{H}^{2m+1}}.$$
For $m=0$, using H\"{o}lder's inequality, interpolation inequalities and Lemma \ref{lem2.2}, we obtain
\begin{align*}
\begin{aligned}
&\langle u_t, (b\cdot\nabla b)_t\rangle\\
=&\langle u_t, b_t\cdot\nabla b+b\cdot\nabla b_t\rangle\\
\lesssim&\|u_t\|_{L^4}\|b_t\|_{L^4}\|\nabla b\|_{L^2}
+\|u\|_{L^4}\|b_t\|_{L^4}\|\nabla b_t\|_{L^2}\\
\lesssim&\|u_t\|_{L^2}^{\frac{1}{2}}\|\nabla u_t\|_{L^2}^{\frac{1}{2}}
\|b_t\|_{L^2}^{\frac{1}{2}}\|\nabla b_t\|_{L^2}^{\frac{1}{2}}\|\nabla \partial_\theta b\|_{L^2}
+\|u_t\|_{L^2}^{\frac{1}{2}}\|\nabla u_t\|_{L^2}^{\frac{1}{2}}
\|\partial_\theta b\|_{L^2}^{\frac{1}{2}}\|\nabla \partial_\theta b\|_{L^2}^{\frac{1}{2}}\|\nabla b_t\|_{L^2}.
\end{aligned}
\end{align*}
Thus,
\begin{align*}
&\int_0^t\langle u_\tau, (b\cdot\nabla b)_\tau\rangle\md\tau\\
\lesssim&\sup_{0\leq\tau\leq t}\|u_\tau\|_{L^2}^{\frac{1}{2}}\|b_\tau\|_{L^2}^{\frac{1}{2}}
\big(\int_0^t\|u_\tau\|_{\dot{H}^1}^2\md\tau\big)^{\frac{1}{4}}
\big(\int_0^t\|b_\tau\|_{\dot{H}^1}^2\md\tau\big)^{\frac{1}{4}}
\big(\int_0^t\|\partial_\theta b\|_{\dot{H}^1}^2\md\tau\big)^{\frac{1}{2}}\\
&+\sup_{0\leq\tau\leq t}\|u_\tau\|_{L^2}^{\frac{1}{2}}
\|\partial_\theta b\|_{L^2}^{\frac{1}{2}}
\big(\int_0^t\|u_\tau\|_{\dot{H}^1}^2\md\tau\big)^{\frac{1}{4}}
\big(\int_0^t\|\partial_\theta b\|_{\dot{H}^1}^2\md\tau\big)^{\frac{1}{4}}
\big(\int_0^t\|b_\tau\|_{\dot{H}^1}^2\md\tau\big)^{\frac{1}{2}}\\
\lesssim& e_0(t)^{\frac{3}{2}}.
\end{align*}
Combining the above two cases, we finally obtain the estimate of the first term of $N_{11}$:
\begin{align}\label{N11}
\int_0^t\langle\Delta^mu_\tau, \Delta^m(b\cdot\nabla b)_\tau\rangle\md\tau
\lesssim e_0(t)^{\frac{3}{2}}.
\end{align}
The remaining terms of $N_{11}$ can be treated in a similar way. We will omit the details.

For the second part of $N_1$, namely
\begin{align*}
N_{12}
=\langle\nabla^{2m+1}u_t, \nabla^{2m+1}(-u\cdot\nabla u+b\cdot\nabla b)\rangle
:=N_{121}+N_{122}.
\end{align*}
Using the divergence-free condition $\nabla\cdot u=0$, Lemma \ref{lem2.1}, Sobolev imbedding inequality and Lemma \ref{lem2.2}, one has
\begin{align*}
\begin{aligned}
N_{121}
=&-\langle\nabla^{2m+1}u_t, \nabla^{2m+2}(u\otimes u)\rangle\\
\lesssim&\|u_t\|_{\dot{H}^{2m+1}}\|u\|_{L^\infty}\|u\|_{\dot{H}^{2m+2}}\\
\lesssim&\|u_t\|_{\dot{H}^{2m+1}}\|\partial_\theta u\|_{L^2}^{\frac{2}{3}}\|\partial_\theta u\|_{\dot{H}^3}^{\frac{1}{3}}
\|\partial_\theta u\|_{\dot{H}^{2m+1}}^{\frac{2}{3}}
\|u\|_{\dot{H}^{2m+4}}^{\frac{1}{3}}\\
=&C(1+t)^{-\frac{1}{3}+\frac{\sigma}{6}}
\|u_t\|_{\dot{H}^{2m+1}}
\|\partial_\theta u\|_{L^2}^{\frac{2}{3}}
(1+t)^{\frac{1}{3}}\|\partial_\theta u\|_{\dot{H}^3}^{\frac{1}{3}}
\|\partial_\theta u\|_{\dot{H}^{2m+1}}^{\frac{2}{3}}
(1+t)^{-\frac{\sigma}{6}}\|u\|_{\dot{H}^{2m+4}}^{\frac{1}{3}}.
\end{aligned}
\end{align*}
Like the estimate of $N_{121}$, $N_{122}$ can be bounded by
\begin{align*}
\begin{aligned}
N_{122}
=&\langle\nabla^{2m+1}u_t, \nabla^{2m+2}(b\otimes b)\rangle\\
\lesssim&\|u_t\|_{\dot{H}^{2m+1}}\|b\|_{L^\infty}\|b\|_{\dot{H}^{2m+2}}\\
\lesssim&\|u_t\|_{\dot{H}^{2m+1}}\|\partial_\theta b\|_{L^2}^{\frac{2}{3}}
\|\partial_\theta b\|_{\dot{H}^3}^{\frac{1}{3}}\|b\|_{\dot{H}^{2m+2}}\\
=&C(1+t)^{-\frac{1}{3}+\frac{\sigma}{3}}
\|u_t\|_{\dot{H}^{2m+1}}
\|\partial_\theta b\|_{L^2}^{\frac{2}{3}}
(1+t)^{\frac{1}{3}}\|\partial_\theta b\|_{\dot{H}^3}^{\frac{1}{3}}
\|b\|_{\dot{H}^{2m+2}}^{\frac{1}{3}}
(1+t)^{-\frac{\sigma}{3}}\|b\|_{\dot{H}^{2m+2}}^{\frac{2}{3}}.
\end{aligned}
\end{align*}
By combining $N_{121}$ and $N_{122}$ and applying H\"{o}lder's inequality in the variable $t$, we obtain
\begin{align}\label{N12}
\begin{aligned}
&\int_0^tN_{12}\md\tau\\
\lesssim&\sup_{0\leq\tau\leq t}(1+\tau)^{-\frac{\sigma}{6}}\|u\|_{H^{2s+4}}^{\frac{1}{3}}
\|\partial_\theta u\|_{L^2}^{\frac{2}{3}}
\big(\int_0^t\|u_\tau\|_{\dot{H}^{2m+1}}^2\md\tau\big)^{\frac{1}{2}}
\big(\int_0^t(1+\tau)^2\|\partial_\theta u\|_{\dot{H}^3}^2\md\tau\big)^{\frac{1}{6}}\\
&\times\big(\int_0^t\|\partial_\theta u\|_{\dot{H}^{2m+1}}^2\md\tau\big)^{\frac{1}{3}}\\
&+\sup_{0\leq\tau\leq t}
\|\partial_\theta b\|_{L^2}^{\frac{2}{3}}\|b\|_{\dot{H}^{2m+2}}^{\frac{1}{3}}
\big(\int_0^t\|u_\tau\|_{\dot{H}^{2m+1}}^2\md\tau\big)^{\frac{1}{2}}
\big(\int_0^t(1+\tau)^2\|\partial_\theta b\|_{\dot{H}^3}^2\md\tau\big)^{\frac{1}{6}}\\
&\times\big(\int_0^t\|\nabla b\|_{H^{2s+1}}^2\md\tau\big)^{\frac{1}{3}}\\
\lesssim&E_0(t)^{\frac{1}{6}}e_0(t)^{\frac{7}{6}}e_1(t)^{\frac{1}{6}}
+E_0(t)^{\frac{1}{3}}e_0(t)e_1(t)^{\frac{1}{6}}.
\end{aligned}
\end{align}

Next, for the term $N_2$, we also divide it into two parts:
\begin{align*}
N_{21}=&-\langle\Delta^{2m+1}u, F_t-\partial_\theta G\rangle-\langle\Delta^{2m+1}b, G_t-\partial_\theta F\rangle,\\
N_{22}=&\langle\Delta^{2m+1}u, \Delta F\rangle.
\end{align*}
For the first term of $N_2$, namely
\begin{align*}
\begin{aligned}
N_{21}
=&-\langle\nabla^{2m+1}u, \nabla^{2m+1}(b\cdot\nabla b)_t\rangle
+\langle\nabla^{2m+1}u,\nabla^{2m+1}(u\cdot\nabla u)_t\rangle\\
&+\langle\nabla^{2m+1}u, \nabla^{2m+1}\partial_\theta(b\cdot\nabla u)\rangle
-\langle\nabla^{2m+1}u,\nabla^{2m+1}\partial_\theta(u\cdot\nabla b)\rangle\\
&-\langle\nabla^{2m+1}b, \nabla^{2m+1}(b\cdot\nabla u)_t\rangle
+\langle\nabla^{2m+1}b,\nabla^{2m+1}(u\cdot\nabla b)_t\rangle\\
&+\langle\nabla^{2m+1}b, \nabla^{2m+1}\partial_\theta(b\cdot\nabla b)\rangle
-\langle\nabla^{2m+1}b,\nabla^{2m+1}\partial_\theta(u\cdot\nabla u)\rangle.
\end{aligned}
\end{align*}
Similar to $N_{11}$, we only need to consider the first term of $N_{21}$.
Using the divergence-free condition $\nabla\cdot b=0$, integration by parts, Lemma \ref{lem2.1}, interpolation inequalities and Lemma \ref{lem2.2}, we have
\begin{align*}
&-\langle\nabla^{2m+1}u, \nabla^{2m+1}(b\cdot\nabla b)_t\rangle\\
=&\langle\nabla^{2m+2}u, \nabla^{2m+1}(b\otimes b)_t\rangle\\
\lesssim&\|u\|_{\dot{H}^{2m+2}}\big(\|b\|_{\dot{H}^{2m+1}}\|b_t\|_{L^\infty}
+\|b\|_{L^\infty}\|b_t\|_{\dot{H}^{2m+1}}\big)\\
\lesssim&\|\partial_\theta u\|_{\dot{H}^{2m+1}}^{\frac{2}{3}}
\|u\|_{\dot{H}^{2m+4}}^{\frac{1}{3}}
\big(\|\partial_\theta b\|_{\dot{H}^{2m+1}}\|b_t\|_{L^2}^{\frac{2}{3}}
\|b_t\|_{\dot{H}^3}^{\frac{1}{3}}
+\|\partial_\theta b\|_{L^2}^{\frac{2}{3}}\|\partial_\theta b\|_{\dot{H}^3}^{\frac{1}{3}}
\|b_t\|_{\dot{H}^{2m+1}}\big)\\
=&C(1+t)^{-\frac{1}{3}+\frac{\sigma}{6}}
\|\partial_\theta u\|_{\dot{H}^{2m+1}}^{\frac{2}{3}}
(1+t)^{-\frac{\sigma}{6}}\|u\|_{\dot{H}^{2m+4}}^{\frac{1}{3}}\\
&\times
\big(\|\partial_\theta b\|_{\dot{H}^{2m+1}}
\|b_t\|_{L^2}^{\frac{2}{3}}
(1+t)^{\frac{1}{3}}\|b_t\|_{\dot{H}^3}^{\frac{1}{3}}
+\|\partial_\theta b\|_{L^2}^{\frac{2}{3}}
(1+t)^{\frac{1}{3}}\|\partial_\theta b\|_{\dot{H}^3}^{\frac{1}{3}}
\|b_t\|_{\dot{H}^{2m+1}}\big).
\end{align*}
Hence, the first term of $N_{21}$ can be bounded by
\begin{align}\label{N21}
\begin{aligned}
&\int_0^t-\langle\nabla^{2m+1}u, \nabla^{2m+1}(b\cdot\nabla b)_\tau\rangle\md\tau\\
\lesssim&\sup_{0\leq\tau\leq t}
(1+\tau)^{-\frac{\sigma}{6}}\|u\|_{H^{2s+4}}^{\frac{1}{3}}
\|b_\tau\|_{L^2}^{\frac{2}{3}}
\big(\int_0^t\|\partial_\theta u\|_{\dot{H}^{2m+1}}^2\md\tau\big)^{\frac{1}{3}}
\big(\int_0^t\|\partial_\theta b\|_{\dot{H}^{2m+1}}^2\md\tau\big)^{\frac{1}{2}}\\
&\times\big(\int_0^t(1+\tau)^2\|b_\tau\|_{\dot{H}^3}^2\md\tau\big)^{\frac{1}{6}}\\
&+\sup_{0\leq\tau\leq t}
(1+\tau)^{-\frac{\sigma}{6}}\|u\|_{H^{2s+4}}^{\frac{1}{3}}
\|\partial_\theta b\|_{L^2}^{\frac{2}{3}}
\big(\int_0^t\|\partial_\theta u\|_{\dot{H}^{2m+1}}^2\md\tau\big)^{\frac{1}{3}}
\big(\int_0^t(1+\tau)^2\|\partial_\theta b\|_{\dot{H}^3}^2\md\tau\big)^{\frac{1}{6}}\\
&\times\big(\int_0^t\|b_\tau\|_{\dot{H}^{2m+1}}^2\md\tau\big)^{\frac{1}{2}}\\
\lesssim&E_0(t)^{\frac{1}{6}}e_0(t)^{\frac{7}{6}}e_1(t)^{\frac{1}{6}}.
\end{aligned}
\end{align}
The remaining terms of $N_{21}$ can be estimated in a similar way and we will omit the details.

For the second term of $N_2$, namely
\begin{align*}
N_{22}=\langle\Delta^{m+1}u, \Delta^{m+1}(b\cdot\nabla b-u\cdot\nabla u)\rangle
:=N_{221}+N_{222}.
\end{align*}
Using the divergence-free condition $\nabla\cdot b=0$, Lemma \ref{lem2.1}, interpolation inequalities and Lemma \ref{lem2.2}, one gets
\begin{align*}
N_{221}
=&\langle\nabla^{2m+2}u, \nabla^{2m+3}(b\otimes b)\rangle\\
\lesssim&\|u\|_{\dot{H}^{2m+2}}\|b\|_{L^\infty}\|b\|_{\dot{H}^{2m+3}}\\
\lesssim&\|\partial_\theta u\|_{\dot{H}^{2m+1}}^{\frac{2}{3}}\|u\|_{\dot{H}^{2m+4}}^{\frac{1}{3}}
\|\partial_\theta b\|_{L^2}^{\frac{2}{3}}\|\partial_\theta b\|_{\dot{H}^3}^{\frac{1}{3}}
\|b\|_{\dot{H}^{2m+3}}\\
=&C(1+t)^{-\frac{1}{3}+\frac{\sigma}{2}+\frac{\sigma}{6}}\|\partial_\theta u\|_{\dot{H}^{2m+1}}^{\frac{2}{3}}
(1+t)^{-\frac{\sigma}{6}}\|u\|_{\dot{H}^{2m+4}}^{\frac{1}{3}}
\|\partial_\theta b\|_{L^2}^{\frac{2}{3}}
(1+t)^{\frac{1}{3}}\|\partial_\theta b\|_{\dot{H}^3}^{\frac{1}{3}}\\
&\times(1+t)^{-\frac{\sigma}{2}}\|b\|_{\dot{H}^{2m+3}}.
\end{align*}
Similarly, $N_{222}$ can be bounded as follows:
\begin{align*}
N_{222}
=&\langle\nabla^{2m+2}u, -[\nabla^{2m+2},u\cdot\nabla]u\rangle\\
\lesssim&\|u\|_{\dot{H}^{2m+2}}^2\|\nabla u\|_{L^\infty}\\
\lesssim&\|\partial_\theta u\|_{\dot{H}^{2m+1}}\|u\|_{\dot{H}^{2m+3}}
\|\partial_\theta u\|_{\dot{H}^1}^{\frac{1}{2}}\|\partial_\theta u\|_{\dot{H}^3}^{\frac{1}{2}}\\
=&C(1+t)^{-\frac{1}{2}+\frac{\sigma}{2}}\|\partial_\theta u\|_{\dot{H}^{2m+1}}
(1+t)^{-\frac{\sigma}{2}}\|u\|_{\dot{H}^{2m+3}}
\|\partial_\theta u\|_{\dot{H}^1}^{\frac{1}{2}}
(1+t)^{\frac{1}{2}}\|\partial_\theta u\|_{\dot{H}^3}^{\frac{1}{2}}.
\end{align*}
We infer from the combined estimates of $N_{221}$ and $N_{222}$ that
\begin{align}\label{N22}
\begin{aligned}
&\int_0^tN_{22}\md\tau\\
\lesssim&\sup_{0\leq\tau\leq t}(1+\tau)^{-\frac{\sigma}{6}}\|u\|_{H^{2s+4}}^{\frac{1}{3}}
\|\partial_\theta b\|_{L^2}^{\frac{2}{3}}
\big(\int_0^t\|\partial_\theta u\|_{\dot{H}^{2m+1}}^2\md\tau\big)^{\frac{1}{3}}
\big(\int_0^t(1+\tau)^2\|\partial_\theta b\|_{\dot{H}^3}^2\md\tau\big)^{\frac{1}{6}}\\
&\times\big(\int_0^t(1+\tau)^{-\sigma}\|\nabla b\|_{H^{2s+2}}^2\md\tau\big)^{\frac{1}{2}}\\
&+\sup_{0\leq\tau\leq t}
(1+\tau)^{-\frac{\sigma}{2}}\|u\|_{H^{2s+3}}
\big(\int_0^t\|\partial_\theta u\|_{\dot{H}^{2m+1}}^2\md\tau\big)^{\frac{1}{2}}
\big(\int_0^t\|\partial_\theta u\|_{\dot{H}^1}^2\md\tau\big)^{\frac{1}{4}}\\
&\times\big(\int_0^t(1+\tau)^2\|\partial_\theta u\|_{\dot{H}^3}^2\md\tau\big)^{\frac{1}{4}}\\
\lesssim&E_0(t)^{\frac{2}{3}}e_0(t)^{\frac{2}{3}}e_1(t)^{\frac{1}{6}}
+E_0(t)^{\frac{1}{2}}e_0(t)^{\frac{3}{4}}e_1(t)^{\frac{1}{4}}.
\end{aligned}
\end{align}
\textbf{Step 5}\quad

Integrating (\ref{3.5}) in time, and summing up the above estimates, that is, (\ref{N11}), (\ref{N12}), (\ref{N21}) and (\ref{N22}), we obtain
\begin{align*}
&\sup_{0\leq\tau\leq t}
\big(\|u_\tau,b_\tau\|_{\dot{H}^{2m}}^2
+\|\partial_\theta u,\partial_\theta b\|_{\dot{H}^{2m}}^2+\|u,b\|_{\dot{H}^{2m+2}}^2\big)\\
&+\int_0^t\big(\|u_\tau,b_\tau\|_{\dot{H}^{2m+1}}^2
+\|\partial_\theta u,\partial_\theta b\|_{\dot{H}^{2m+1}}^2\big)\md\tau\\
\lesssim&\varepsilon^2+e_0(t)^{\frac{3}{2}}
+E_0(t)^{\frac{1}{6}}e_0(t)^{\frac{7}{6}}e_1(t)^{\frac{1}{6}}
+E_0(t)^{\frac{1}{3}}e_0(t)e_1(t)^{\frac{1}{6}}
+E_0(t)^{\frac{2}{3}}e_0(t)^{\frac{2}{3}}e_1(t)^{\frac{1}{6}}\\
&+E_0(t)^{\frac{1}{2}}e_0(t)^{\frac{3}{4}}e_1(t)^{\frac{1}{4}}.
\end{align*}
We complete the proof of this lemma.
\end{proof}

Now, we establish the decay estimate $e_1(t)$.
\begin{lem}\label{lem3.4}
From the definition in (\ref{2.3}), for $1\leq m\leq s-1$, $s\geq 2$, we have
\begin{align}\label{3}
e_1(t)\lesssim\varepsilon^2+e_0(t)+e_0(t)^{\frac{1}{2}}e_1(t)
+E_0(t)^{\frac{17}{120}}e_0(t)^{\frac{1}{3}}e_1(t)^{\frac{41}{40}}.
\end{align}
\end{lem}

\begin{proof}
This lemma can be proved in a similar fashion as the proof of Lemma \ref{lem3.3}.
The difference is that after multiplying by the time-weight $(1+t)^2$, compared to (\ref{3.5}), there will be some linear terms on the right-hand side of (\ref{3.6}).
By utilizing Gagliardo-Nirenberg's inequality and Young's inequality, we can effectively control these terms.

Following the same steps as deriving the equation (\ref{3.5}), we obtain
\begin{align}\label{3.6}
\begin{aligned}
&\frac{\md}{\md t}(1+t)^2
\big(\|u_t,b_t\|_{\dot{H}^{2m}}^2+\|\partial_\theta u,\partial_\theta b\|_{\dot{H}^{2m}}^2
+\|u,b\|_{\dot{H}^{2m+2}}^2\big)\\
&+(1+t)^2\big(\|u_t, b_t\|_{\dot{H}^{2m+1}}^2
+\|\partial_\theta u,\partial_\theta b\|_{\dot{H}^{2m+1}}^2\big)
=\sum_{i=1}^{i=5}M_i,
\end{aligned}
\end{align}
where,
\begin{align*}
M_1=&(1+t)\|u_t,b_t\|_{\dot{H}^{2m}}^2,\\
M_2=&(1+t)\|\partial_\theta u,\partial_\theta b\|_{\dot{H}^{2m}}^2,\\
M_3=&(1+t)\|u,b\|_{\dot{H}^{2m+2}}^2,\\
M_4=&(1+t)^2\langle\Delta^{2m}u_t, F_t-\Delta F-\partial_\theta G\rangle
+(1+t)^2\langle\Delta^{2m}b_t, G_t-\partial_\theta F\rangle,\\
M_5=&-(1+t)^2\langle\Delta^{2m+1}u, F_t-\Delta F-\partial_\theta G\rangle
-(1+t)^2\langle\Delta^{2m+1}b, G_t-\partial_\theta F\rangle.
\end{align*}
We begin with the extra linear terms $M_1-M_3$ caused by the time-weight.
By Gagliardo-Nirenberg's interpolation and Young's inequality, we have
\begin{align}\label{3.9}
\begin{aligned}
\int_0^tM_1\md\tau
\leq &C\int_0^t(1+\tau)\|u_\tau,b_\tau\|_{\dot{H}^{2m-1}}\|u_\tau,b_\tau\|_{\dot{H}^{2m+1}}\md\tau\\
\leq&\frac{1}{8}\int_0^t(1+\tau)^2\|u_\tau,b_\tau\|_{\dot{H}^{2m+1}}^2\md\tau
+C\int_0^t\|u_\tau,b_\tau\|_{\dot{H}^{2m-1}}^2\md\tau\\
\leq&\frac{1}{8}\int_0^t(1+\tau)^2\|u_\tau,b_\tau\|_{\dot{H}^{2m+1}}^2\md\tau+Ce_0(t),
\end{aligned}
\end{align}
provided $m\geq1$.
Similarly, one has
\begin{align}\label{3.10}
\begin{aligned}
\int_0^tM_2\md\tau
\leq\frac{1}{8}\int_0^t(1+\tau)^2\|\partial_\theta u,\partial_\theta b\|_{\dot{H}^{2m+1}}^2\md\tau+Ce_0(t).
\end{aligned}
\end{align}
From the interpolation inequality, Young's inequality, and Lemma \ref{lem2.2}, we get
\begin{align}\label{3.11}
\begin{aligned}
\int_0^tM_3\md\tau
\leq&C\int_0^t(1+\tau)\|u,b\|_{\dot{H}^{2m+1}}\|u,b\|_{\dot{H}^{2m+3}}\md\tau\\
\leq&\frac{1}{8}\int_0^t(1+\tau)^2\|u,b\|_{\dot{H}^{2m+1}}^2\md\tau
+C\int_0^t\|u,b\|_{\dot{H}^{2m+3}}^2\md\tau\\
\leq&\frac{1}{8}\int_0^t(1+\tau)^2\|\partial_\theta u,\partial_\theta b\|_{\dot{H}^{2m+1}}^2\md\tau+Ce_0(t),
\end{aligned}
\end{align}
provided $m\leq s-1$.

The estimation of nonlinear terms is similar to $N_1$ and $N_2$ in Lemma \ref{lem3.3}. 
In addition, the key step here is to allocate time weights reasonably.
First of all, we divide $M_4$ into two parts:
\begin{align*}
M_{41}=&(1+t)^2\langle\Delta^{2m}u_t, F_t-\partial_\theta G\rangle
+(1+t)^2\langle\Delta^{2m}b_t, G_t-\partial_\theta F\rangle,\\
M_{42}=&(1+t)^2\langle\Delta^{2m}u_t, -\Delta F\rangle.
\end{align*}
For the first part of $M_4$, namely
\begin{align*}
M_{41}
=&(1+t)^2\langle\Delta^mu_t, \Delta^m(b\cdot\nabla b)_t\rangle
-(1+t)^2\langle\Delta^mu_t,\Delta^m(u\cdot\nabla u)_t\rangle\\
&-(1+t)^2\langle\Delta^mu_t, \Delta^m\partial_\theta(b\cdot\nabla u)\rangle
+(1+t)^2\langle\Delta^mu_t,\Delta^m\partial_\theta(u\cdot\nabla b)\rangle\\
&+(1+t)^2\langle\Delta^mb_t, \Delta^m(b\cdot\nabla u)_t\rangle
-(1+t)^2\langle\Delta^mb_t,\Delta^m(u\cdot\nabla b)_t\rangle\\
&-(1+t)^2\langle\Delta^mb_t, \Delta^m\partial_\theta(b\cdot\nabla b)\rangle
+(1+t)^2\langle\Delta^mb_t,\Delta^m\partial_\theta(u\cdot\nabla u)\rangle,
\end{align*}
we also only estimate the first term of $M_{41}$.
Using integration by parts, Lemma \ref{lem2.1}, Sobolev imbedding theorem and Lemma \ref{lem2.2}, for $m\geq 1$, one has
\begin{align*}
\begin{aligned}
&\langle\Delta^mu_t, \Delta^m(b\cdot\nabla b)_t\rangle\\
=&-\langle\nabla^{2m+1}u_t, \nabla^{2m}(b\otimes b)_t\rangle\\
\lesssim&\|u_t\|_{\dot{H}^{2m+1}}\big(\|b\|_{\dot{H}^{2m}}\|b_t\|_{L^\infty}
+\|b\|_{L^\infty}\|b_t\|_{\dot{H}^{2m}}\big)\\
\lesssim&\|u_t\|_{\dot{H}^{2m+1}}\big(\|\partial_\theta b\|_{\dot{H}^{2m-2}}^{\frac{1}{3}}
\|\partial_\theta b\|_{\dot{H}^{2m+1}}^{\frac{2}{3}}
\|b_t\|_{L^2}^{\frac{2}{3}}\|b_t\|_{\dot{H}^3}^{\frac{1}{3}}
+\|\partial_\theta b\|_{L^2}^{\frac{2}{3}}\|\partial_\theta b\|_{\dot{H}^3}^{\frac{1}{3}}
\|b_t\|_{\dot{H}^{2m-2}}^{\frac{1}{3}}
\|b_t\|_{\dot{H}^{2m+1}}^{\frac{2}{3}}\big).
\end{aligned}
\end{align*}
Hence, the first term of $M_{41}$ can be bounded by
\begin{align}\label{M41}
\begin{aligned}
&\int_0^t(1+\tau)^2\langle\Delta^mu_\tau, \Delta^m(b\cdot\nabla b)_\tau\rangle\md\tau\\
\lesssim&\sup_{0\leq\tau\leq t}\|\partial_\theta b\|_{\dot{H}^{2m-2}}^{\frac{1}{3}}
\|b_\tau\|_{L^2}^{\frac{2}{3}}
\big(\int_0^t(1+\tau)^2\|u_\tau\|_{\dot{H}^{2m+1}}^2\md\tau\big)^{\frac{1}{2}}\\
&\times\big(\int_0^t(1+\tau)^2\|\partial_\theta b\|_{\dot{H}^{2m+1}}^2\md\tau\big)^{\frac{1}{3}}
\big(\int_0^t(1+\tau)^2\|b_\tau\|_{\dot{H}^3}^2\md\tau\big)^{\frac{1}{6}}\\
&+\sup_{0\leq\tau\leq t}\|b_\tau\|_{\dot{H}^{2m-2}}^{\frac{1}{3}}
\|\partial_\theta b\|_{L^2}^{\frac{2}{3}}
\big(\int_0^t(1+\tau)^2\|u_\tau\|_{\dot{H}^{2m+1}}^2\md\tau\big)^{\frac{1}{2}}\\
&\times\big(\int_0^t(1+\tau)^2\|\partial_\theta b\|_{\dot{H}^3}^2\md\tau\big)^{\frac{1}{6}}
\big(\int_0^t(1+\tau)^2\|b_\tau\|_{\dot{H}^{2m+1}}^2\md\tau\big)^{\frac{1}{3}}\\
\lesssim& e_0(t)^{\frac{1}{2}}e_1(t).
\end{aligned}
\end{align}
The remaining terms of $M_{41}$ can be estimated in a similar manner with appropriate adjustments. Further details will be omitted.

For the second term of $M_4$, namely
\begin{align*}
M_{42}=(1+t)^2\langle\nabla^{2m+1}u_t, \nabla^{2m+1}(u\cdot\nabla u-b\cdot\nabla b)\rangle
:=M_{421}+M_{422}.
\end{align*}
Using the divergence-free condition $\nabla\cdot u=0$, Lemma \ref{lem2.1} and Lemma \ref{lem2.2}, one has
\begin{align*}
M_{421}=
&(1+t)^2\langle\nabla^{2m+1}u_t, \nabla^{2m+2}(u\otimes u)\rangle\\
\lesssim&(1+t)^2\|u_t\|_{\dot{H}^{2m+1}}\|u\|_{L^\infty}\|u\|_{\dot{H}^{2m+2}}\\
\lesssim&(1+t)^2\|u_t\|_{\dot{H}^{2m+1}}\|\partial_\theta u\|_{L^2}^{\frac{2}{3}}
\|\partial_\theta u\|_{\dot{H}^3}^{\frac{1}{3}}
\|\partial_\theta u\|_{\dot{H}^{2m+1}}^{\frac{2}{3}}\|u\|_{\dot{H}^{2m+4}}^{\frac{1}{3}}\\
=&C(1+t)\|u_t\|_{\dot{H}^{2m+1}}
\|\partial_\theta u\|_{L^2}^{\frac{2}{3}}
(1+t)^{\frac{1}{3}}\|\partial_\theta u\|_{\dot{H}^3}^{\frac{1}{3}}
(1+t)^{\frac{2}{3}}\|\partial_\theta u\|_{\dot{H}^{2m+1}}^{\frac{2}{3}}
\|u\|_{\dot{H}^{2m+4}}^{\frac{1}{3}}.
\end{align*}
Similarly, $M_{422}$ can be bounded as follows:
\begin{align*}
M_{422}=&(1+t)^2\langle\nabla^{2m+1}u_t, \nabla^{2m+2}(-b\otimes b)\rangle\\
\lesssim&(1+t)^2\|u_t\|_{\dot{H}^{2m+1}}\|b\|_{L^\infty}\|b\|_{\dot{H}^{2m+2}}\\
\lesssim&(1+t)^2\|u_t\|_{\dot{H}^{2m+1}}\|\partial_\theta b\|_{L^2}^{\frac{2}{3}}
\|\partial_\theta b\|_{\dot{H}^3}^{\frac{1}{3}}
\|\partial_\theta b\|_{\dot{H}^{2m+1}}^{\frac{2}{3}}\|b\|_{\dot{H}^{2m+4}}^{\frac{1}{3}}\\
=&C(1+t)\|u_t\|_{\dot{H}^{2m+1}}
\|\partial_\theta b\|_{L^2}^{\frac{2}{3}}
(1+t)^{\frac{1}{3}}\|\partial_\theta b\|_{\dot{H}^3}^{\frac{1}{3}}
(1+t)^{\frac{2}{3}}\|\partial_\theta b\|_{\dot{H}^{2m+1}}^{\frac{2}{3}}
\|b\|_{\dot{H}^{2m+4}}^{\frac{1}{3}}.
\end{align*}
Therefore, combining the estimates of $M_{421}$ and $M_{422}$ gives that
\begin{align}\label{M42}
\begin{aligned}
\int_0^tM_{42}\md\tau
\lesssim&\sup_{0\leq\tau\leq t}
\|\partial_\theta u\|_{L^2}^{\frac{2}{3}}
\|u\|_{\dot{H}^{2m+4}}^{\frac{1}{3}}
\big(\int_0^t(1+\tau)^2\|u_\tau\|_{\dot{H}^{2m+1}}^2\md\tau\big)^{\frac{1}{2}}\\
&\times\big(\int_0^t(1+\tau)^2\|\partial_\theta u\|_{\dot{H}^3}^2\md\tau\big)^{\frac{1}{6}}
\big(\int_0^t(1+\tau)^2\|\partial_\theta u\|_{\dot{H}^{2m+1}}^2\md\tau\big)^{\frac{1}{3}}\\
&+\sup_{0\leq\tau\leq t}
\|\partial_\theta b\|_{L^2}^{\frac{2}{3}}
\|b\|_{\dot{H}^{2m+4}}^{\frac{1}{3}}
\big(\int_0^t(1+\tau)^2\|u_\tau\|_{\dot{H}^{2m+1}}^2\md\tau\big)^{\frac{1}{2}}\\
&\times\big(\int_0^t(1+\tau)^2\|\partial_\theta b\|_{\dot{H}^3}^2\md\tau\big)^{\frac{1}{6}}
\big(\int_0^t(1+\tau)^2\|\partial_\theta b\|_{\dot{H}^{2m+1}}^2\md\tau\big)^{\frac{1}{3}}\\
\lesssim& e_0(t)^{\frac{1}{2}}e_1(t).
\end{aligned}
\end{align}

Next, for the term $M_5$, we also divide it into two parts:
\begin{align*}
M_{51}=&-(1+t)^2\langle\Delta^{2m+1}u, F_t-\partial_\theta G\rangle
-(1+t)^2\langle\Delta^{2m+1}b, G_t-\partial_\theta F\rangle,\\
M_{52}=&-(1+t)^2\langle\Delta^{2m+1}u, -\Delta F\rangle.
\end{align*}
For the first part of $M_5$, namely
\begin{align*}
M_{51}
=&-(1+t)^2\langle\nabla^{2m+1}u, \nabla^{2m+1}(b\cdot\nabla b)_t\rangle
+(1+t)^2\langle\nabla^{2m+1}u,\nabla^{2m+1}(u\cdot\nabla u)_t\rangle\\
&+(1+t)^2\langle\nabla^{2m+1}u, \nabla^{2m+1}\partial_\theta(b\cdot\nabla u)\rangle
-(1+t)^2\langle\nabla^{2m+1}u,\nabla^{2m+1}\partial_\theta(u\cdot\nabla b)\rangle\\
&-(1+t)^2\langle\nabla^{2m+1}b, \nabla^{2m+1}(b\cdot\nabla u)_t\rangle
+(1+t)^2\langle\nabla^{2m+1}b,\nabla^{2m+1}(u\cdot\nabla b)_t\rangle\\
&+(1+t)^2\langle\nabla^{2m+1}b, \nabla^{2m+1}\partial_\theta(b\cdot\nabla b)\rangle
-(1+t)^2\langle\nabla^{2m+1}b,\nabla^{2m+1}\partial_\theta(u\cdot\nabla u)\rangle,
\end{align*}
we only consider the first term of $M_{51}$.
By utilizing the divergence-free condition $\nabla\cdot b=0$, integration by parts, Lemma \ref{lem2.1} and Lemma \ref{lem2.2}, one gets
\begin{align*}
&-\langle\nabla^{2m+1}u, \nabla^{2m+1}(b\cdot\nabla b)_t\rangle\\
=&\langle\nabla^{2m+2}u, \nabla^{2m+1}(b\otimes b)_t\rangle\\
\lesssim&\|u\|_{\dot{H}^{2m+2}}\big(\|b\|_{\dot{H}^{2m+1}}\|b_t\|_{L^\infty}
+\|b\|_{L^\infty}\|b_t\|_{\dot{H}^{2m+1}}\big)\\
\lesssim&\|\partial_\theta u\|_{\dot{H}^{2m+1}}^{\frac{2}{3}}\|u\|_{\dot{H}^{2m+4}}^{\frac{1}{3}}
\big(\|\partial_\theta b\|_{\dot{H}^{2m+1}}\|b_t\|_{L^2}^{\frac{2}{3}}\|b_t\|_{\dot{H}^3}^{\frac{1}{3}}
+\|\partial_\theta b\|_{L^2}^{\frac{2}{3}}\|\partial_\theta b\|_{\dot{H}^3}^{\frac{1}{3}}
\|b_t\|_{\dot{H}^{2m+1}}\big).
\end{align*}
Hence, the first term of $M_{51}$ can be estimated by
\begin{align}\label{M51}
\begin{aligned}
&\int_0^t-(1+\tau)^2\langle\nabla^{2m+1}u, \nabla^{2m+1}(b\cdot\nabla b)_\tau\rangle\md\tau\\
\lesssim&\sup_{0\leq\tau\leq t}\|b_\tau\|_{L^2}^{\frac{2}{3}}\|u\|_{\dot{H}^{2m+4}}^{\frac{1}{3}}
\big(\int_0^t(1+\tau)^2\|\partial_\theta u\|_{\dot{H}^{2m+1}}^2\md\tau\big)^{\frac{1}{3}}\\
&\times\big(\int_0^t(1+\tau)^2\|\partial_\theta b\|_{\dot{H}^{2m+1}}^2\md\tau\big)^{\frac{1}{2}}
\big(\int_0^t(1+\tau)^2\|b_\tau\|_{\dot{H}^3}^2\md\tau\big)^{\frac{1}{6}}\\
&+\sup_{0\leq\tau\leq t}\|\partial_\theta b\|_{L^2}^{\frac{2}{3}}\|u\|_{\dot{H}^{2m+4}}^{\frac{1}{3}}
\big(\int_0^t(1+\tau)^2\|\partial_\theta u\|_{\dot{H}^{2m+1}}^2\md\tau\big)^{\frac{1}{3}}\\
&\times\big(\int_0^t(1+\tau)^2\|\partial_\theta b\|_{\dot{H}^3}^2\md\tau\big)^{\frac{1}{6}}
\big(\int_0^t(1+\tau)^2\|b_\tau\|_{\dot{H}^{2m+1}}^2\md\tau\big)^{\frac{1}{2}}\\
\lesssim& e_0(t)^{\frac{1}{2}}e_1(t).
\end{aligned}
\end{align}
After some adjustments, the remaining terms of $M_{51}$ can be handled by similar methods. We will omit details here.

For the second term of $M_5$, namely
\begin{align*}
M_{52}=(1+t)^2\langle\nabla^{2m+2}u, \nabla^{2m+2}(b\cdot\nabla b-u\cdot\nabla u)\rangle
:=M_{521}+M_{522}.
\end{align*}
For $M_{521}$, the wildest term, we will carefully balance the time-wight $(1+t)^2$ and the index of interpolation inequalities, so that this term can be controlled by the energy framework defined in Section 2.
Using the divergence-free condition, Lemma \ref{lem2.1}, interpolation inequalities and Lemma \ref{lem2.2}, one has
\begin{align}\label{d2}
\begin{aligned}
M_{521}
=&(1+t)^2\langle\nabla^{2m+2}u, \nabla^{2m+3}(b\otimes b)\rangle\\
\lesssim&(1+t)^2\|u\|_{\dot{H}^{2m+2}}\|b\|_{L^\infty}\|b\|_{\dot{H}^{2m+3}}\\
\lesssim&(1+t)^2\|\partial_\theta u\|_{\dot{H}^{2m+2}}^{\frac{1}{4}}\|\partial_\theta u\|_{\dot{H}^{2m+1}}^{\frac{2}{3}}
\|u\|_{\dot{H}^{2m+10}}^{\frac{1}{12}}
\|\partial_\theta b\|_{L^2}^{\frac{2}{3}}\|\partial_\theta b\|_{\dot{H}^3}^{\frac{1}{3}}\|\partial_\theta b\|_{\dot{H}^{2m+1}}^{\frac{4}{5}}
\|b\|_{\dot{H}^{2m+11}}^{\frac{1}{5}}\\
=&C(1+t)^{-\frac{1}{20}+\frac{\sigma}{24}+\frac{\sigma}{10}}
(1+t)^{\frac{1}{4}}\|u\|_{\dot{H}^{2m+2}}^{\frac{1}{4}}
(1+t)^{\frac{2}{3}}\|\partial_\theta u\|_{\dot{H}^{2m+1}}^{\frac{2}{3}}
(1+t)^{-\frac{\sigma}{24}}\|u\|_{\dot{H}^{2m+10}}^{\frac{1}{12}}\\
&\times
\|\partial_\theta b\|_{L^2}^{\frac{2}{3}}
(1+t)^{\frac{1}{3}}\|\partial_\theta b\|_{\dot{H}^3}^{\frac{1}{3}}(1+t)^{\frac{4}{5}}\|\partial_\theta b\|_{\dot{H}^{2m+1}}^{\frac{4}{5}}
(1+t)^{-\frac{\sigma}{10}}\|b\|_{\dot{H}^{2m+11}}^{\frac{1}{5}}.
\end{aligned}
\end{align}
Similarly, $M_{522}$ can be bounded by
\begin{align*}
M_{522}
=&-(1+t)^2\langle\nabla^{2m+2}u, [\nabla^{2m+2},u\cdot\nabla]u\rangle\\
\lesssim&(1+t)^2\|u\|_{\dot{H}^{2m+2}}\|u\|_{\dot{H}^{2m+2}}\|\nabla u\|_{L^\infty}\\
\lesssim&(1+t)^2\|u\|_{\dot{H}^{2m+2}}^{\frac{1}{2}}\|\partial_\theta u\|_{\dot{H}^{2m+1}}
\|u\|_{\dot{H}^{2m+4}}^{\frac{1}{2}}
\|\partial_\theta u\|_{\dot{H}^1}^{\frac{1}{2}}
\|\partial_\theta u\|_{\dot{H}^3}^{\frac{1}{2}}\\
=&C
(1+t)^{\frac{1}{2}}\|u\|_{\dot{H}^{2m+2}}^{\frac{1}{2}}
(1+t)\|\partial_\theta u\|_{\dot{H}^{2m+1}}
\|u\|_{\dot{H}^{2m+4}}^{\frac{1}{2}}
\|\partial_\theta u\|_{\dot{H}^1}^{\frac{1}{2}}
(1+t)^{\frac{1}{2}}\|\partial_\theta u\|_{\dot{H}^3}^{\frac{1}{2}}.
\end{align*}

Hence we deduce, by combining $M_{521}$ and $M_{522}$, that
\begin{align}
\nonumber&\int_0^tM_{52}\md\tau\\
\nonumber\lesssim&\sup_{0\leq\tau\leq t}(1+\tau)^{\frac{1}{4}}\|u\|_{\dot{H}^{2m+2}}^{\frac{1}{4}}
(1+\tau)^{-\frac{\sigma}{24}}\|u\|_{H^{2s+8}}^{\frac{1}{12}}
\|\partial_\theta b\|_{L^2}^{\frac{2}{3}}
\big(\int_0^t(1+\tau)^2\|\partial_\theta u\|_{\dot{H}^{2m+1}}^2\md\tau\big)^{\frac{1}{3}}\\
\nonumber&\times\big(\int_0^t(1+\tau)^2\|\partial_\theta b\|_{\dot{H}^{2m+1}}^2\md\tau\big)^{\frac{2}{5}}
\big(\int_0^t(1+\tau)^2\|\partial_\theta b\|_{\dot{H}^3}^2\md\tau\big)^{\frac{1}{6}}
\big(\int_0^t(1+\tau)^{-\sigma}\|\nabla b\|_{H^{2s+8}}^2\md\tau\big)^{\frac{1}{10}}\\
\nonumber&+\sup_{0\leq\tau\leq t}
(1+\tau)^{\frac{1}{2}}\|u\|_{\dot{H}^{2m+2}}^{\frac{1}{2}}
\|u\|_{\dot{H}^{2m+4}}^{\frac{1}{2}}
\big(\int_0^t(1+\tau)^2\|\partial_\theta u\|_{\dot{H}^{2m+1}}^2\md\tau\big)^{\frac{1}{2}}
\big(\int_0^t\|\partial_\theta u\|_{\dot{H}^1}^2\md\tau\big)^{\frac{1}{4}}\\
\nonumber&\times
\big(\int_0^t(1+\tau)^2\|\partial_\theta u\|_{\dot{H}^3}^2\md\tau\big)^{\frac{1}{4}}\\
\lesssim& E_0(t)^{\frac{17}{120}}e_0(t)^{\frac{1}{3}}e_1(t)^{\frac{41}{40}}
+e_0(t)^{\frac{1}{2}}e_1(t).\label{M52}
\end{align}

Integrating (\ref{3.6}) in time and invoking the estimates for $M_1$-$M_5$, that is (\ref{3.9})-(\ref{M51}) and (\ref{M52}), we deduce that
\begin{align*}
\begin{aligned}
&\sup_{0\leq\tau\leq t}(1+\tau)^2
\big(\|u_\tau,b_\tau\|_{\dot{H}^{2m}}^2+\|\partial_\theta u,\partial_\theta b\|_{\dot{H}^{2m}}^2
+\|u,b\|_{\dot{H}^{2m+2}}^2\big)\\
&+\int_0^t(1+\tau)^2\big(\|u_\tau,b_\tau\|_{\dot{H}^{2m+1}}^2
+\|\partial_\theta u,\partial_\theta b\|_{\dot{H}^{2m+1}}^2\big)\md\tau\\
\lesssim&\varepsilon^2+e_0(t)+e_0(t)^{\frac{1}{2}}e_1(t)
+E_0(t)^{\frac{17}{120}}e_0(t)^{\frac{1}{3}}e_1(t)^{\frac{41}{40}}.
\end{aligned}
\end{align*}
We complete the proof of this lemma.
\end{proof}

\section{Proof of Theorem \ref{thm0.1}}

In this section, we utilize the bootstrapping argument to conclude the proof of Theorem \ref{thm0.1}.

\emph{Proof of Theorem \ref{thm0.1}.}
The existence of local smooth solutions follows a standard procedure (refer to the book \cite{T2001, MB2002} for more information).

Thus, our only goal is to prove
\begin{eqnarray*}
E(t)=E_0(t)+e_0(t)+e_1(t)\leq C\varepsilon^2
\end{eqnarray*}
holds for all $t>0$. Here, $E(t)$ denotes the total energy and $C$ represents a positive constant defined later.
At the outset, it is evident from the setting of initial data (\ref{**}) that $E(0)\lesssim\varepsilon^2$.
By multiplying $(\ref{1})$, $(\ref{2})$, and $(\ref{3})$ with a suitable constant respectively, and then adding them together, we obtain
\begin{align}\label{4}
E(t)\leq C_0\varepsilon^2+C_0E(t)^{\frac{3}{2}}.
\end{align}

Next, by setting $\varepsilon:=\frac{1}{4C_0^{3/2}}$ and making the ansatz that
\begin{align*}
E(t)\leq \frac{1}{4C_0^2},
\end{align*}
we conclude, from equation (\ref{4}), that
\begin{align*}
E(t)\leq2C_0\varepsilon^2:=\frac{1}{8C_0^2},
\end{align*}
for some positive constant $C_0$.

Consequently, the bootstrapping argument demonstrates that
\begin{align*}
E(t)\leq C\varepsilon^2,\quad\forall t\in[0,\infty),
\end{align*}
where $C=2C_0$.

The proof of Theorem \ref{thm0.1} is now complete.

\vspace{.2in}

\textbf{Acknowledgements}
The author sincerely thank Prof. Yi zhou and Prof. Yi zhu for helpful discussions.

\end{document}